\documentclass[11pt]{amsart}
\usepackage{amscd,amsfonts,mathrsfs,amssymb,amsmath}
\usepackage{hyperref}
\usepackage[margin=3.6 cm]{geometry}
\usepackage{verbatim}
\usepackage{xcolor}
\usepackage{graphicx}
\usepackage{ifpdf}

\def\H{\mathbb H}

\newtheorem{thm}{Theorem}[section]
\newtheorem{lem}{Lemma}[section]
\newtheorem{rmk}{Remark}[section]
\newtheorem{prp}{Proposition}
\newtheorem{cor}{Corollary}[section]

\newtheorem{dfn}[thm]{Definition}

\newcommand{\be}{\begin{equation}}
\newcommand{\ee}{\end{equation}}
\newcommand{\bea}{\begin{eqnarray}}
\newcommand{\eea}{\end{eqnarray}}
\newcommand{\Bea}{\begin{eqnarray*}}
\newcommand{\Eea}{\end{eqnarray*}}
\newcommand{\bt}{\begin{Theorem}}
\newcommand{\et}{\end{Theorem}}

\begin{document}

\title[Recurrent and irregular orbits of the horocyclic flow on the unit tangent bundle of the untwisted flute]{Recurrent and irregular orbits of the horocyclic flow on the unit tangent bundle of the untwisted flute} 

\author[Amadou Sy]{Amadou Sy}
\address{Department of mathematics and computer science at the Cheikh Anta Diop University of Dakar (Senegal)}
\email{amadou22.sy@ucad.edu.sn}

\begin{abstract}
The aim of this article is to show that if there exists  \( u \in \Omega_{h} \subset T^{1}S \) an infinite quasi-minimizing ray which do not intersect any closed geodesic on the surface $S$ (untwisted flute), then
$T_{u}=\{ t \in \mathbb{R} \; ; \; g_{t}u \in \overline{h_{\mathbb{R}}u} \}=\{0\}.
$

\end{abstract}
\footnote{
AMS 2010 \textit{Mathematics Subject Classification.} Primary 37D40; Secondary 20H10, 14H55, 30F35.}

\maketitle
\section{Introduction}
Let $S$ be an untwisted flute surface and $
u \in \Omega_{h} \subset T^{1}S$ where $\Omega_{h}$ denotes the non-wandering set of the horocycle flow. 
One way to understand the non-minimality of the horocycle flow is to study the intersection 
between the closure of the horocycle orbit and the geodesic orbit of $u$. 
The main results of this paper are the following :

\begin{thm}\label{pa}
Let \( S\) be an untwiste flute and \( u \in \Omega_{h} \subset T^{1}S \) such that the horocyclic orbit \( h_{\mathbb{R}}u \) is neither closed nor dense in the non-wandering set of the horocyclic flow.  
If $u(\mathbb{R^{+}})$ does not intersect any closed  geodesic on the surface $S$, then  
$$
T_{u} = \{ t \in \mathbb{R} \; ; \; g_{t}u \in \overline{h_{\mathbb{R}}u} \} = \{0\}.
$$
\end{thm}

This result is of fundamental in the sense that it allows us to highlight the question raised by Bellis in his thesis concerning the flute \( F_{4} \) see (\cite{BellisThesis}).
Let's consider $u \in \Omega_{h} \subset T^{1}S $ such that $h_{\mathbb{R}}u$ is not closed nor dense in $\Omega_{h}.$ If $u(\mathbb{R^{+}})$ does not intersect any close geodesic on the surface then can we have $g_{\mathbb{R}}u \cap \overline{h_{\mathbb{R}}u} = \{0\}?
$
Here, $S$ is a geometrically infinite hyperbolic surface, and $h_{\mathbb{R}}$ and $g_{\mathbb{R}}$ denote the horocyclic and geodesic flows on $T^{1}S$, respectively.  

This theorem generalizes the example of the flute $F_{4}$. Recall that the same phenomenon occurs in the case where the horocyclic orbit is closed (see \cite{BellisThesis}).  
This may appear surprising, but it is due to the special behavior of a certain infinite quasi-minimizing ray on the surface, which does not intersect any closed geodesic.  

As shown by the following theorem:

\begin{thm}\label{TAHA}
There exists a untwisted flute \( F \) and a vector \( u \in \Omega_{h} \subset T^{1}F \), corresponding to an infinite quasi-minimizing geodesic ray, such that:
\begin{enumerate}
    \item The horocyclic orbit \( h_{\mathbb{R}}u \) is neither closed nor dense;
    \item \( \mathrm{Inj}\, u(\mathbb{R}^{+}) = +\infty \);
    \item \( T_{u} = \{ t \in \mathbb{R} \; ; \; g_{t}u \in \overline{h_{\mathbb{R}}u} \} = \{0\}. \)
\end{enumerate}
\end{thm}
This document is structured as follows. We begin by recalling some useful results in the first part. Then, in the second part, we prove the main theorem, and after we will give also the proof of Theorem~\ref{TAHA}, which is an application of the main theorem. And finally we will present some properties of $T_{u}.$

\section{Preliminaries}

The study of infinite-type hyperbolic surfaces, particularly twisted and untwisted flutes, dates back to the works of Ara Basmajian and Andrew Haas (see \cite{AB} ).  
We may adopt the following definition of a flute:

\begin{dfn}
A hyperbolic surface \( S = \Gamma \setminus \mathbb{H} \) is called a \emph{flute} if \( \Gamma \) is an infinite Schottky group generated by a sequence \( (\gamma_n)_{n \geq 1} \) of hyperbolic isometries whose axes are nested.
\end{dfn}

We distinguish two types of flutes: the \textbf{twisted} and the \textbf{untwisted}. To define them, let us denote for each \( n \) by \( a_n \) the hyperbolic segment joining the axis of \( \gamma_n \) to that of \( \gamma_{n+1} \), orthogonal to both.
\begin{dfn}[Untwisted and Twisted Flutes]
A flute \( F \) is said to be \textbf{untwisted} if the geodesic segments \( a_n \) all lie on the same half-geodesic \( \tilde{u}(\mathbb{R}^+) \) in \( \mathbb{H} \). Otherwise, \( F \) is said to be \textbf{twisted}.
\end{dfn}



\section{Convergence in $T^{1}S$, Asymptotic Thinness, and Limit Points}

In this section, we recall some results on convergence in $ T^{1}S$ and on the asymptotic thinness of geodesic half-rays, which will be necessary to prove Theorem \ref{pa}.

Let $S=\Gamma\setminus\mathbb{H}$ be a hyperbolic surface, where $\Gamma$ is a Fuchsian group with no elliptic elements.

For $x \in S$ and  $u \in T^1S$, we denote by $\widetilde{x}$ and $\widetilde{u}$ arbitrary lifts of $x$ and $u$ in $\mathbb{H}$ and $T^1\mathbb{H}$, respectively.

Let \( (\widetilde{u}(t))_{t \geq 0} \) be the unit-speed parametrization of the geodesic half-ray \( \widetilde{u}(\mathbb{R}_{+}) \subset \mathbb{H} \), defined by \( \widetilde{u} \), with origin \( \widetilde{u}(0) \) and endpoint at infinity \( \widetilde{u}(\infty) \).

We denote by \( u(\mathbb{R}_{+}) \) the geodesic half-ray on \( S \) defined by \( u \), and by \( (u(t))_{t \geq 0} \) its projection onto \( S \), that is, the projection of \( (\widetilde{u}(t))_{t \geq 0} \).
\subsection{Characterization of the Elements of $T_{u}$}

Let \( u \in T^1S \) and \( \widetilde{u} \) be a lift of \( u \) to \( T^1\mathbb{H} \). For \( z \in \mathbb{H} \), denote by \( B_{\widetilde{u}(\infty)}(\widetilde{u}(0), z) \) the Busemann cocycle centered at \( \widetilde{u}(\infty) \), evaluated at \( \widetilde{u}(0) \) and \( z \).  

The set \( \{ z \in \mathbb{H} \mid B_{\widetilde{u}(\infty)}(\widetilde{u}(0), z) = 0 \} \) is the horocycle centered at \( \widetilde{u}(\infty) \) and passing through \( \widetilde{u}(0) \).  
(For properties of the Busemann cocycle, see \cite{Dalbo2007}.)

We now show that determining whether a real number \( t \) belongs to \( T_u \) is equivalent to constructing a sequence \( (\gamma_n)_{n \geq 0} \) of elements in \( \Gamma \) satisfying two conditions.  
This is based on the horocyclic convergence proved in \cite{BellisThesis}.
\newpage
More precisely:

\begin{prp}\cite{BellisThesis},( Proposition 2.3.2)
Let \( u \) and \( v \) be two elements of \( T^1S \). Then \( v \in \overline{h_{\mathbb{R}}(u)} \) if and only if there exists a sequence \( (\gamma_n)_{n \geq 0} \subset \Gamma \) such that the following two conditions are satisfied:
\begin{itemize}
    \item[\textnormal{(i)}] \( \displaystyle\lim_{n \to +\infty} \gamma_n \widetilde{u}(\infty) = \widetilde{v}(\infty) \).

    \item[\textnormal{(ii)}] \( \displaystyle\lim_{n \to +\infty} B_{\widetilde{u}(\infty)}(\gamma_n^{-1}i, \widetilde{u}(0)) = B_{\widetilde{v}(\infty)}(i, \widetilde{v}(0)) \).
\end{itemize}

\end{prp}

For \( t \neq 0 \), if we set \( v = g_t u \), then there exists \( \alpha \in \Gamma \) such that \( \widetilde{v}(\infty) = \alpha \widetilde{u}(\infty) \) and \( \widetilde{v}(0) = \widetilde{g}_t \alpha \widetilde{u}(0) \).  
The points (i) and (ii) from the previous proposition then become:

\[
\textnormal{(i)'} \quad \lim_{n \to +\infty} \gamma_n \widetilde{u}(\infty) = \alpha \widetilde{u}(\infty),
\]

\[
\textnormal{(ii)'} \quad \lim_{n \to +\infty} B_{\widetilde{u}(\infty)}(\gamma_n^{-1}i, \widetilde{u}(0)) = B_{\alpha \widetilde{u}(\infty)}(i, \alpha \widetilde{g}_t \widetilde{u}(0)) = B_{\widetilde{u}(\infty)}(\alpha^{-1} i, \widetilde{u}(0)) + t.
\]

We then obtain the following corollary:

\begin{cor}\label{tunv}
Let \( u \in T^1S \). Then a nonzero real number \( t \) belongs to the set \( T_u \) if and only if there exist \( \alpha \in \Gamma \) and a sequence \( (\gamma_n)_{n \geq 0} \subset \Gamma \) such that the following two conditions hold:
\begin{itemize}
    \item[\textnormal{(i)}] \( \displaystyle\lim_{n \to +\infty} \gamma_n \widetilde{u}(\infty) = \alpha \widetilde{u}(\infty), \)
    \item[\textnormal{(ii)}] \( \displaystyle\lim_{n \to +\infty} B_{\widetilde{u}(\infty)}(\gamma_n^{-1} i, \alpha^{-1} i) = t. \)
\end{itemize}
\end{cor}

\begin{cor}\label{MM}
A non-periodic horocyclic orbit \( h_{\mathbb{R}} u \) is recurrent if and only if there exists a sequence \( (\gamma_n)_{\mathbb{N}} \) of pairwise distinct elements of \( \Gamma \) such that:
\begin{enumerate}
    \item \( \displaystyle\lim_{n \to +\infty} \gamma_n \widetilde{u}(\infty) = \widetilde{u}(\infty) \);
    \item \( \displaystyle\lim_{n \to +\infty} B_{\widetilde{u}(\infty)}(\gamma_n^{-1} i, i) = 0 \).
\end{enumerate}
\end{cor}
\subsection{Asymptotic Thinness}

We recall here only the results that will be needed later. For more details on this notion, the reader may refer to \cite{SyGaye2025}.

\begin{dfn}
Let \( x \in S \) and \( u \in T^1S \).

The injectivity radius of \( S \) at \( x \) is the quantity defined by:
\[
\mathrm{Inj}(x) = \inf_{\gamma \in \Gamma \setminus \{ \mathrm{Id} \}} d(\widetilde{x}, \gamma \widetilde{x}),
\]
where \( \widetilde{x} \) is any lift of \( x \) to \( \mathbb{H} \).

The \emph{asymptotic thinness} of the geodesic half-ray \( u(\mathbb{R}_{+}) \) is the quantity denoted by \( \mathrm{Inj}(u(\mathbb{R}_{+})) \), defined as:
\[
\mathrm{Inj}(u(\mathbb{R}_{+})) = \liminf_{t \to +\infty} \mathrm{Inj}(u(t)).
\]
\end{dfn}
\newpage
\begin{lem}\label{T}
Let $F = \mathbb{H}/\Gamma$ be an untwisted flute. Let $u\in\Omega_{h}\subset T^{1}S$ a quasi-minimizing ray such that $\tilde{u}(\mathbb{R^{+}})\cap(\gamma_{n}^{-};\gamma_{n}^{+})=\emptyset$. Let $(\alpha; \beta)$ be the orthogonal commun of the axe of $(\gamma_{n})_{n}$ denoted by $(\gamma_{n}^{-};\gamma_{n}^{+})$ , where $\Gamma = \langle \gamma_n;n\in\mathbb{N} \rangle $ is a geometrically infinite Fuchsian group. Then :
\begin{enumerate}
    \item If \( (\alpha; \beta) = (0; \infty) \), then \( a_n = d_n \) for all \( n \in \mathbb{N} \).
    
    \item If \( \alpha \) and \( \beta \) are real, then
    $
    (a_n - d_n)(\alpha + \beta) + 2b_n = 2\alpha \beta c_n \quad \text{for all } n \in \mathbb{N}.
    $
    
    \item If \( (\alpha; \beta) = (\alpha; +\infty) \) with \( \alpha \in \mathbb{R} \), then
    $
    a_n - d_n = 2c_n \quad \text{for all } n \in \mathbb{N}.
    $
\end{enumerate}
\end{lem}

\begin{proof}
We consider the twisted flute obtained as the quotient of the unit tangent bundle of \( \mathbb{H} \) by the Fuchsian group \( \Gamma \) generated by the \( \gamma_n \).

Since we are dealing with a twisted flute, the axis \( (\alpha; \beta) \) is orthogonal to the axis of \( \gamma_n \), that is, to \( (\gamma_n^{-}; \gamma_n^{+}) \).

A straightforward computation yields the result by applying Lemma \ref{QI}.
\end{proof}

\subsection{Nature of the points of the limit set}
Let $\Gamma$ be a fuchsian group without elleptic element, and    $\Sigma=\Gamma\backslash\mathbb{H}$ be the associated hyperbolic surface. Let us denote here 
 $\Lambda=\overline{\Gamma.i}\setminus\Gamma.i$ the limit set of the group $\Gamma$. In $\Lambda$ we distinguish four types of points according to the orbit $\Gamma.i$ intersects the open horodisks based at this limit point (see \textbf{\cite{DalboStarkov2000}}).

- A point $\eta$ of $\Lambda$ is said to be horocyclic if for any open horodisk $O_{\eta}$ based in $\eta$, the set $\Gamma.i \cap O_{\eta}$ is infinite. We denote $\Lambda_h $for all of these points.

- A point $\eta$ of $\Lambda$ is said to be discrete if for any open horodisk $O_{\eta}$ based in $\eta$, the set $\Gamma.i \cap O_{\eta}$ is finite. We denote $\Lambda_d $ for all of these points.

- A point $\eta$ of $\Lambda$ is said to be parabolic if it is fixed by a parabolic isometry of $\Gamma$. We denote $\Lambda_p$ for all of these points.
- A point $\eta$ of $\Lambda$ is said to be irregular if it is neither horocyclic, nor discrete and neither
parabolic. We note $\Lambda_{irr}$ for all of these points.

The sets $\Lambda_h, \Lambda_d, \Lambda_p$ and $\Lambda_{irr}$ form an invariant partition $\Gamma$ of $\Lambda$.

\begin{rmk}{plhd}
Let be $\infty$ a point in the limit set $\Lambda$.

- $\infty\notin\Lambda_h$ if and only if there exists $M>0$ such for all $\gamma\in\Gamma, Im(\gamma(i))\leq M.$

- $\infty\notin\Lambda_d$  if there exists a non-constant sequence $(Im(\gamma_n(i)))_{n\geq 1}$ which tends towards $l>0$.
\end{rmk}
The following well-known theorem (see for example \textbf{\cite{DalboStarkov2000}}) gives a characterization of the geometric finiteness of the surface as a function of the limit points.
\begin{thm}
The surface $\Sigma=\Gamma\setminus\H$ is geometrically finite if and only if $\Lambda= \Lambda_h\cup\Lambda_p$. 
\end{thm}
\subsection{Correspondence between limit points and the topological nature of orbits.}

For $u \in T^1\Sigma$, we denote by $\widetilde{u}$ a lift of $u$, and by $u^{+}$ the positive endpoint of the geodesic in $\mathbb{H}^2$ defined by $\widetilde{u}$. 
It is well known that $u$ belongs to $\Omega_h$ if and only if the endpoint $u^{+}$ belongs to the limit set $\Lambda$. 
There is thus a correspondence between the topological nature of the orbit $h_{\mathbb{R}}(u)$ and the nature of the limit point $u^{+}$ in $\Lambda$ (see for instance \textbf{\cite{Starkov1995}} and \textbf{\cite{DalboStarkov2000}}). 
More precisely, we have the following correspondence:

{\it 
\begin{itemize}
  \item the orbit $h_{\mathbb{R}}(u)$ is dense in $\Omega_h$ if and only if $u^{+} \in \Lambda_h$;
  \item the orbit $h_{\mathbb{R}}(u)$ is periodic if and only if $u^{+} \in \Lambda_p$;
  \item the orbit $h_{\mathbb{R}}(\pi(u))$ is closed and non-periodic if and only if $u^{+} \in \Lambda_d$;
  \item the orbit $h_{\mathbb{R}}(\pi(u))$ is irregular (that is, neither closed nor dense in $\Omega_h$) if and only if $u^{+} \in \Lambda_{irr}$.
\end{itemize}
}


We have the following decomposition, which can be found in \cite{DalboStarkov2000}:

\begin{prp}
Let \( \Gamma \) be a geometrically infinite Fuchsian group.  
Then
\[
\Lambda_{\Gamma} = \Lambda_h \cup \Lambda_d \cup \Lambda_p \cup \Lambda_{\mathrm{irr}}.
\]
\end{prp}

\begin{prp}[\cite{GayeLo2017}]
The group \( \Gamma \) is geometrically infinite if and only if \( \Lambda_d \cup \Lambda_{\mathrm{irr}} \neq \emptyset \).
\end{prp}
\begin{lem}\label{A}
Let \( \Gamma \) be a geometrically infinite Fuchsian group.  
Let \( u \in \Omega_h \subset T^1S \) such that \( \widetilde{u}(\infty) \in \Lambda_{\mathrm{irr}} \).  
Then there exist constants \( b > a > 0 \) and a sequence \( (\gamma_n)_n \) of elements in \( \Gamma \) such that:
\begin{enumerate}
    \item For all \( n > 0 \), \( \operatorname{Im}(\gamma_n i) \in (a, b) \),
    \item \( \gamma_n(i) \to \infty \).
\end{enumerate}
\end{lem}

\begin{proof}
Up to conjugating the group \( \Gamma \) by an element of \( PSL(2, \mathbb{R}) \), we may assume that \( \widetilde{u}(\infty) = \infty \in \Lambda_{\mathrm{irr}} \). Since \( \infty \notin \Lambda_h \), there exists \( b' > 0 \) such that

$\left( \Gamma.i \cap \operatorname{Int}(O_{\infty}(ib')) \right) < \infty,$
where \( \operatorname{Int}(O_{\infty}(ib')) \) denotes the interior of the horodisk centered at \( \infty \) and passing through \( ib' \).  
That is, \( \operatorname{Int}(O_{\infty}(ib')) = \{ z \in \mathbb{H} \; ; \; \operatorname{Im}(z) > b' \} \).

Hence, except for finitely many elements \( \gamma_1, \gamma_2, \ldots, \gamma_p \in \Gamma \), we have:
$
\gamma(i) \notin \operatorname{Int}(O_{\infty}(ib')).
$
Thus, We have $\operatorname{Im}(\gamma(i)) \leq b' $. By setting 
$
b = \max \left( b', \max_{1 \leq k \leq p} \operatorname{Im}(\gamma_k i) \right),
$
we get for all \( \gamma \in \Gamma \),
$
\operatorname{Im}(\gamma(i)) \leq b.
$
Since \( \infty \notin \Lambda_d \), there exists \( a > 0 \) such that 
$
\# \left( \Gamma.i \cap \operatorname{Int}(O_{\infty}(i a)) \right) = +\infty.
$

Thus, there exist infinitely many elements \( \gamma \in \Gamma \) such that
$
\gamma(i) \in \operatorname{Int}(O_{\infty}(i a)).
$

As \( \Gamma \) contains no elliptic elements, we can assume the isometries \( \gamma \) are pairwise distinct.

Hence, there exist infinitely many distinct elements \( \gamma \) such that
$
a \leq \operatorname{Im}(\gamma(i)) \leq b.
$
Therefore, the set 
$
M = \{ \gamma(i) \mid a \leq \operatorname{Im}(\gamma(i)) \leq b \}
$
is infinite. Since \( \Gamma \) is discrete, this set admits an accumulation point
$
x \in \partial \mathbb{H} = \mathbb{R} \cup \{\infty\}.
$

Consequently, there exists a sequence of pairwise distinct elements \( (\gamma_n) \subset \Gamma \) such that
$
\gamma_n i \to x,
$
with
$
a \leq \operatorname{Im}(\gamma_n i) \leq b,
$
so that
$
\operatorname{Im}(\gamma_n i) \geq a > 0.
$

Hence, \( x \in \mathbb{R} \cup \{\infty\} \) and necessarily
$
x = \infty.
$
\end{proof}
\begin{lem}\label{H}
Let \(\Gamma\) be a geometrically infinite Fuchsian group.  
Let \(u \in \Omega_{h} \subset T^{1}S\) such that the orbit \(h_{\mathbb{R}} u\) is not closed nor dense in $\Omega_{h}$.  
Then there exists a sequence \((\gamma_n)_{n} \subset \Gamma\) of pairwise distinct elements such that \(\gamma_n(\infty) \to \infty\).
\end{lem}

\begin{proof}
By the previous lemma, there exists a sequence \((\alpha_n)_n\) of pairwise distinct elements in \(\Gamma\) and constants \(b \geq a > 0\) such that:  
\begin{enumerate}
    \item \(\alpha_n i \to \infty\),
    \item \(a \leq \operatorname{Im}(\alpha_n i) \leq b\).
\end{enumerate}

For each \(n \in \mathbb{N}\), write  
$
\alpha_n z = \frac{a_n z + b_n}{c_n z + d_n}
$
with \(a_n d_n - b_n c_n = 1\).  
We necessarily have \(c_n \neq 0\), otherwise \(\alpha_n(\infty) = \infty\), which would imply that \(\infty\) is a horocyclic or parabolic limit point, which is excluded.  
Thus, we can write  
$
\alpha_n(\infty) = \frac{a_n}{c_n}.
$
Since \( \operatorname{Im}(\alpha_n i) = \frac{1}{c_n^2 + d_n^2} \in (a,b) \), the sequences \((|c_n|)_n\) and \((|d_n|)_n\) are bounded.  
By extracting subsequences if necessary, we have \(c_n \to c\) and \(d_n \to d\), with  
$
0 < \frac{1}{b} \leq c^2 + d^2 \leq \frac{1}{a}.
$  
Let \((\alpha_n)\) be this subsequence such that \(c_n \to c\), \(d_n \to d\), and \(\alpha_n i \to \infty\) with \(a \leq \operatorname{Im}(\alpha_n i) \leq b\).  
We have the real part of \(\alpha_n i\):  
$
\operatorname{Re}(\alpha_n i) = \frac{a_n}{c_n} - \frac{d_n}{c_n (c_n^2 + d_n^2)}.
$
\begin{enumerate}
    \item If \(c_n \to c \in \mathbb{R}^*\), then  
    $
    \frac{d_n}{c_n (c_n^2 + d_n^2)} \to \frac{d}{c(c^2 + d^2)} \in \mathbb{R},
    $
    and since \(\alpha_n i \to \infty\) and \(\operatorname{Im}(\alpha_n i) \in (a,b)\), it follows that  
    $
    \operatorname{Re}(\alpha_n i) \to \infty,
    $
    thus \(\alpha_n(\infty) \to \infty\).
    
    \item If \(c_n \to 0\),  
    suppose that the sequence \((a_n)_n\) admits a subsequence converging to 0. Denote this subsequence again by \((a_n)_n\). Then  
    $
    \alpha_n^{-1} i = \frac{1}{a_n^2 + c_n^2} \to \infty.
    $
    Thus, for every \( B > 0 \), there exists \( N \in \mathbb{N} \) such that for all \( n \geq N \),  
$
\operatorname{Im}(\alpha_n^{-1} i) > B,
$  
which implies that  
$
\alpha_n^{-1} i \in \operatorname{Int}(O_{\infty}(iB)).
$  
Consequently, we have  
$\# \big(\Gamma i \cap \operatorname{Int}(O_{\infty}(iB))\big) = +\infty,
$  
so that \(\infty \in \Lambda_h\), which is absurd. Therefore, the sequence \((a_n)_n\) admits a subsequence converging to \( l \in \mathbb{R}^* \cup \{\infty\} \).
\end{enumerate}

\end{proof}

From points on the boundary, we can define a geometric invariant that plays a very important role.
\begin{dfn}
Let \( a, b, c, d \) be four distinct points on the boundary at infinity of \( \mathbb{H} \).  
The \emph{cross-ratio} of these points is the quantity defined by:
\[
[a; b; c; d] = \frac{(a - c)(b - d)}{(a - d)(b - c)}.
\]
\end{dfn}

\begin{lem}[\cite{Beardon1983}]\label{QI}
Let \( (a, b) \) and \( (c, d) \) be two geodesics intersecting at a point \( x \).  
If the boundary points are ordered as \( (a; c; b; d) \), then the angle \( \beta \in [0, \pi] \) between the geodesics \( (a, b) \) and \( (c, d) \) is given by:
\[
[a, c, d, b] = \frac{\cos(\beta) + 1}{2}.
\]
\end{lem}
\begin{lem}\label{TAHA2}

Let $(a,b,c,d)$ be four points at infinity such that $(a,b)$ is not interlaced with $(c,d)$ and $(a,c,d,b)$ is ordered.  
Then there exists a unique geodesic $\gamma$ orthogonal to both $(a,b)$ and $(c,d)$.  

Moreover, we have:
\begin{equation}
  [a,c,d,b] \;=\; \tfrac{1}{2} \cosh^{2}\!\left(\tfrac{d(x,y)}{2}\right),
\end{equation}
\begin{equation}
  [a,b,c,d] \;=\; \tanh^{2}\!\bigl(d(x,y)\bigr),
\end{equation}
where $x$ is the intersection of $\gamma$ with $(a,b)$ and $y$ is the intersection of $\gamma$ with $(c,d)$.

\end{lem}
\begin{thm}\label{GB}
    Let \(P\) be a hyperbolic polygon with \(n\) sides and interior angles \(\alpha_{1}, \alpha_{2}, \ldots, \alpha_{n}\). Then the area of this polygon is given by:
    \[
    \text{Area}(P) = \pi(n - 2) - \sum_{i=1}^{n} \alpha_{i}.
    \]
\end{thm}

\section{Proof of Theorem \ref{pa}}

Let \( F \) be a untwisted flute and let \( \Gamma = \langle \gamma_n; \, n \in \mathbb{N} \rangle \) be a geometrically infinite Fuchsian group.  
Let \( u \in T^1F \) such that the horocyclic orbit \( h_{\mathbb{R}}u \) is neither closed nor dense.  
According to the characterization by Dal'bo and Starkov (see \cite{Starkov1995}), the positive endpoint of a lift of \( u \), denoted by \( \widetilde{u}(\infty) \), belongs to \( \Lambda_{\mathrm{irr}} \).  

We now rely on the following two lemmas:

Assuming that none of the closed geodesics $\alpha^{*}_n$ intersect $\tilde{u}(\mathbb{R}^+)$, and without loss of generality, let us set $(-\alpha; +\alpha)$ as the common orthogonal to the axes $\alpha_n^{*}$, with $\alpha > 0$.
By Lemmas \ref{T}, \ref{A}, and \ref{H}, we can extract a subsequence $(\gamma_{n})\subset \Gamma$, which we will still denote by \(\gamma_n\), such that their coefficients \((a_n, b_n, c_n, d_n)\) satisfy the following relations:
\begin{enumerate}
    \item \(a_n \rightarrow \frac{1}{d}, \quad b_n \rightarrow 0, \quad c_n \rightarrow 0, \quad d_n \rightarrow d \neq 0.\)
    \item \(a_n \rightarrow \frac{1 + \alpha^2 c^2}{d}, \quad b_n \rightarrow -\alpha^2 c, \quad c_n \rightarrow c \neq 0, \quad d_n \rightarrow d \neq 0.\)
    \item \(a_n \rightarrow ?, \quad b_n \rightarrow -\alpha^2 c, \quad c_n \rightarrow c \neq 0, \quad d_n \rightarrow 0.\)
\end{enumerate}
\textbf{First case}:\\
Suppose the group coefficients satisfy: 
\[ a_n \rightarrow \frac{1}{d}, \quad b_n \rightarrow 0, \quad c_n \rightarrow 0, \quad d_n \rightarrow d \neq 0. \]
Using the fact that the group is discrete and applying the proposition on horocyclic convergence, that is (see \textbf{\cite{BellisThesis}}, Chapter 2, Proposition 2.3.2), we have \(\gamma_n(\infty) \rightarrow \infty\) and \(B_{\infty}(\gamma_n^{-1}(i); i) \rightarrow 0\). 
Thus, we deduce that the orbit is recurrent according to Corollary \ref{MM} and therefore \(T_u = \{0\}.\)\\
\textbf{Second case}: \\
Suppose that 
\[
a_n \rightarrow \frac{1+\alpha^{2} c^{2}}{d}, \quad b_n \rightarrow -\alpha^{2} c, \quad c_n \rightarrow c \neq 0, \quad d_n \rightarrow d \neq 0,
\]
then the sequence of isometries \(\gamma_n\) converges to an element of \(\Gamma\). This is absurd since \(\Gamma\) is discrete.\\
\textbf{Third case} \\
Suppose that these subsequences \((a_n)_n\), \((b_n)_n\), \((c_n)_n\) and \((d_n)_n\) satisfy condition (3). \\ 
\begin{enumerate}
    \item If \(\mathrm{Inj} u ( \mathbb{R}^+ ) < \infty\), then the sequence \((a_n)_n\) converges to some \(a\). This again contradicts the assumption that \(\Gamma\) is discrete.
    \item If \(\mathrm{Inj} u ( \mathbb{R}^+ ) = +\infty\), then the limit of the translation lengths of the axes of \(\gamma_n\) is infinite, that is 
    \[
    \lim_{n \to \infty} l(\gamma_n) = +\infty,
    \]
    and thus, according to Proposition $1.1.2$ in \cite{BellisThesis}, we deduce that the sequence \((a_n)_n\) tends to \(+\infty\).
\end{enumerate}
Hence, we have:
\[
a_n \to +\infty, \quad b_n \to -\alpha^2 c, \quad c_n \to c \neq 0, \quad d_n \to 0.
\]
Let us consider the unique geodesic $(\beta_{n},\frac{x_{n}}{2})$ passing in $\frac{x_{n}}{2}$ and orthogonal to the axis of $\gamma_n$ which denote here by $(y_{n},x_{n})$. By applying the lemma \ref{QI}, we can get $\beta_{n}=\frac{-x_{n}^{2}}{2y_{n}}+\frac{3x_{n}}{2}$. Let $\theta_n$ be the angle between the geodesics $(\beta_{n},\frac{x_{n}}{2})$ and $(0;\infty).$ 
The points \(A\), \(B\), \(C\) and $D$ denote the intersection points respectively of the following geodesics: ($(\beta_{n}; \frac{x_{n}}{2})$ and $(0; \infty)$), ($(-\beta_{n}; +\frac{x_{n}}{2})$ and $(\gamma_{n}^{-}; \gamma_{n}^{+})$), and ($(-\alpha; +\alpha)$ and $(0; \infty)$) and  ($(-\alpha; +\alpha)$ and $(\gamma_{n}^{-}; \gamma_{n}^{+})$). (See \ref{III})

\begin{figure}\label{III}
  \centering
\includegraphics[width=1\linewidth]{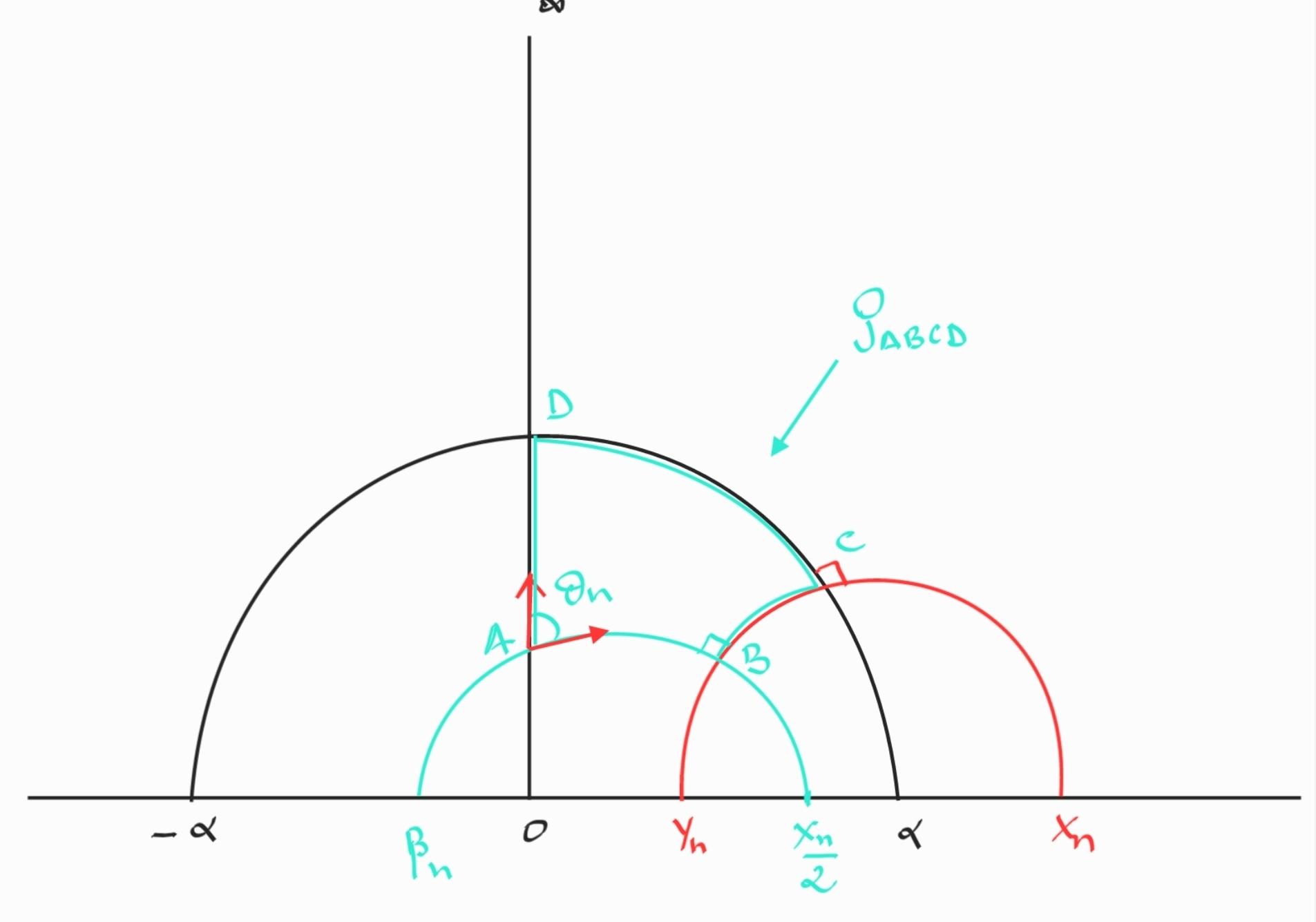}
    \caption{Hyperbolic polygon}

\end{figure}

Applying the Gauss-Bonnet theorem \ref{GB} to this hyperbolic polygon  $P_{ABCD}$ (see Figure \ref{III}), we obtain $\theta_n \leq \frac{\pi}{2},  \forall n \in \mathbb{N}.$

By applying Lemma \ref{QI} to the oriented geodesics \((\beta_{n}, \frac{x_{n}}{2})\) and $((0; +\infty)$, we obtain:
$
[\beta_{n}; 0; \infty; \frac{x_{n}}{2}] = \frac{1 + \cos(\theta_n)}{2}.
$
Letting \(n \to +\infty\), it follows that \(\theta_n \to \pi\), since \(\gamma_n^{-} \to 0\) and \(\gamma_n^{+} \to +\infty\).

\section{Proof of Theorem \ref{TAHA}}

\subsection{Presentation of the Group and Construction of the Untwisted Flute}

We consider the Poincaré half-plane $\mathbb{H}$ endowed with its hyperbolic metric.\\
Let $H_1$ be the horocycle centered at $\infty$ and passing through the point $i$.\\
Let $\tilde{A}$ denote the geodesic with endpoints $(-1, 1)$, and let $R$ be the reflection with respect to this geodesic.\\ 
Let $H_2$ be the image of $H_{\infty}(i) = H_1$ under the reflection with axis $\tilde{A}$.\\

\begin{figure}
    \centering
    \includegraphics[width=0.9\linewidth]{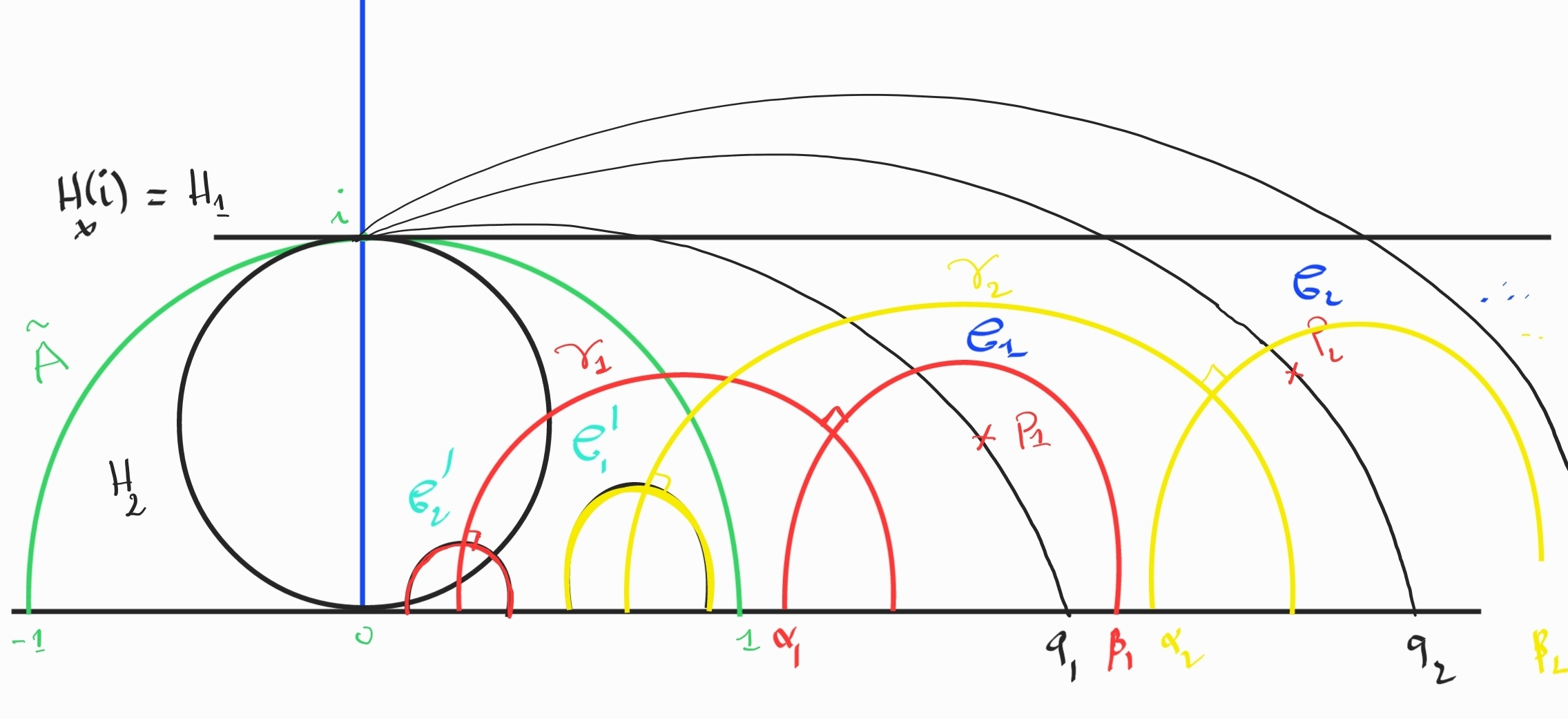}
    \caption{Fundamental domain.}

\end{figure}

Next, consider a sequence of geodesic half-rays $\left(g_n\right)_{n \in \mathbb{N}}$, all starting from the point $i$, with endpoints at infinity $\xi_n > 0$ satisfying $\xi_n \to +\infty$ as $n \to \infty$.\\
We also fix a sequence $\left(\epsilon_n\right)_{n \in \mathbb{N}}$ of positive real numbers converging to $0$.\\

Let $\left(p_n\right)_{n \in \mathbb{N}}$ be a sequence of points such that, for each $n$, the point $p_n$ lies on $g_n$ and satisfies 
\[ B_{\infty}(p_n, i) = \epsilon_n. \]
From this relation we deduce that $p_n = e^{-\epsilon_n}$, that is, $p_n = 1 - \epsilon_n + O(\epsilon_n)$.\\
We therefore set $I_n = 1 - \epsilon_n$, $\mathrm{Im}(p_n) = I_n$, and $Y_n = \mathrm{Re}(p_n)$. 
Let $C_n$ denote the perpendicular bisector of the segment $[i, p_n]$, namely,
\[
C_n = \{ z \in \mathbb{H} \mid d(z, i) = d(z, p_n) \}.
\]
Explicitly, $C_n$ is a circle with center $X_n = \dfrac{Y_n}{I_n - 1}$ and Euclidean radius 
\[
R_n = \sqrt{I_n} \left( 1 + \frac{Y_n^2}{(I_n - 1)^2} \right)^{1/2}.
\]

\begin{prp}
Let $(\alpha_n, \beta_n)$ be the geodesic corresponding to the perpendicular bisector of the segment $[i, p_n]$, where $\alpha_n$ and $\beta_n$ denote its negative and positive endpoints, respectively. Then
\[
\lim_{n \to +\infty} \alpha_n = \lim_{n \to +\infty} \beta_n = +\infty.
\]
\end{prp}

\begin{proof}
Let $Y_n$ and $I_n$ be the real and imaginary parts of $p_n$, respectively.\\
The equation of the perpendicular bisector $C_n$ of $[i, p_n]$ is given by
\[
I_n |z - i|^2 = |z - p_n|^2.
\]
Writing \( z = x + i y \), we obtain
\[
(I_n - 1)x^2 + (I_n - 1)y^2 + 2 Y_n x = I_n^2 - I_n + Y_n^2.
\]
Hence,
\[
\left(x + \frac{Y_n}{I_n - 1}\right)^2 + y^2 = I_n + \frac{Y_n^2}{(I_n - 1)^2},
\]
which shows that $C_n$ is a circle centered at $X_n$ with radius $R_n$, as stated.\\

By definition,
\[
\alpha_n = X_n - R_n = -\frac{Y_n}{I_n - 1} + \sqrt{I_n} \left( 1 + \frac{Y_n^2}{(I_n - 1)^2} \right)^{1/2}
= \frac{Y_n + \sqrt{I_n}\sqrt{(I_n - 1)^2 + Y_n^2}}{I_n - 1}.
\]
This last expression is asymptotically equivalent to $\dfrac{(1 + I_n) Y_n}{I_n - 1}$.\\
We thus conclude that $\lim_{n \to +\infty} \alpha_n = +\infty$, and since $\alpha_n \leq \beta_n$, the same holds for $\beta_n$.
\end{proof}

Since $C_n$ is a circle with center $X_n = \dfrac{Y_n}{I_n - 1}$ and Euclidean radius 
$R_n = \sqrt{I_n}\left( 1 + \dfrac{Y_n^2}{(I_n - 1)^2} \right)^{1/2}$,
there exists $k \in \mathbb{N}$ such that $C_n \cap C_k = \emptyset$. 
Indeed, the circles $C_n$ and $C_k$ are externally disjoint if and only if 
$|X_n - X_k| > R_n + R_k.$\\
Since the sequence $(X_n)_n$ is increasing, defining $K_n = X_n - R_n$, we have $K_{n+1} - K_n > 0.$
We now define the circles $(C'_n)_n$ as the reflections of the circles $(C_n)_n$ across the geodesic $\tilde{A} = (-1,1)$.

\begin{prp}
The sequence $(C'_n)_n$ consists of euclidean circles with centers $X'_n = \dfrac{X_n}{X_n^2 - R_n^2}$ and radiu $K_n = \dfrac{R_n}{|X_n^2 - R_n^2|}.$
\end{prp}

\begin{proof}
It suffices to note that the image of a circle under reflection about a circle is again a circle. 
A straightforward computation of $j(C_n) = \{ j(z) \mid z \in C_n \}$ yields the desired result:
\[
X'_n = \frac{X_n}{X_n^2 - R_n^2}
\quad \text{and} \quad
K_n = \frac{R_n}{|X_n^2 - R_n^2|}.
\]
\end{proof}

\begin{prp}{FIN}
There exists a unique hyperbolic isometry $f_n$ such that $f_n(C_n) = C'_n$.
\end{prp}

\begin{proof}
From the above, we have $j(C_n) = C'_n$, where $j$ is the reflection with respect to $\tilde{A}$. 
Let $R$ be the reflection with respect to $C_n$, that is,
\[
R(z) = X_n + \frac{R_n^2}{\overline{z - X_n}}.
\]
Composing both sides of $j(C_n) = C'_n$ with $R$, we obtain
\[
f_n = j \circ R, \quad \text{so that} \quad f_n(C_n) = C'_n.
\]
Explicitly,
\[
f_n(z) = j \circ R(z) = \frac{z - X_n}{X_n z + R_n^2 - X_n^2}.
\]
Let $M_n$ denote the matrix associated with this real Möbius transformation. 
Its trace equals $R_n^2$.\\
After normalization, the associated matrix becomes
\[
M'_n =
\begin{pmatrix}
\dfrac{1}{R_n} & \dfrac{-X_n}{R_n} \\[4pt]
\dfrac{X_n}{R_n} & \dfrac{R_n^2 - X_n^2}{R_n}
\end{pmatrix}.
\]
We observe that $\operatorname{Tr}(M'_n) \to +\infty$ as $n \to \infty$, so there exists $N \in \mathbb{N}$ such that for all $n \geq N$, we have $\operatorname{Tr}(M'_n) \geq 5$.
\end{proof}

The hyperbolic isometry $f_n$ defined above has an axis perpendicular to both $C_n$ and $C'_n$, and satisfies $f_n(C_n) = C'_n$. 
Thus, for all $n$, we have $d(i, C_n) = d(i, C'_n)$, hence
\[
d(i, C'_n) = d(f_n i, f_n C_n) = d(f_n i, C'_n).
\]
Therefore, for all $n$,
\[
\partial \mathbb{H}_i(f_n) = C'_n
\quad \text{and} \quad
\partial \mathbb{H}_i(f_n^{-1}) = C_n,
\]
where $\mathbb{H}_i(f_n) = \{ z \in \mathbb{H} \mid d(z, i) \leq d(f_n(i), z) \}.$

Moreover, by construction, if $k$ and $l$ are distinct indices, then
\[
\big( \mathbb{H}_i^c(f_k^{-1}) \cup \mathbb{H}_i^c(f_k) \big)
\cap
\big( \mathbb{H}_i^c(f_l^{-1}) \cup \mathbb{H}_i^c(f_l) \big)
= \emptyset.
\]
We deduce that the group $\Gamma := \langle f_n \mid n \geq 1 \rangle$ is an infinite Schottky group generated by a sequence of hyperbolic elements whose axes are nested.\\
Hence, the quotient surface $T := \Gamma \backslash \mathbb{H}$ is an \emph{untwisted flute surface} (see~\ref{GEO}).\\
By construction, the geodesic $\tilde{A}$ intersects all the axes $\left(f_n^-, f_n^+\right)$ orthogonally, and we have
\[
\liminf_{n \to +\infty} \ell(f_n) = +\infty.
\]
Therefore, $T$ is an untwisted flute surface with fundamental domain
\[
D_i(\Gamma) = \bigcap_{n \geq 1} \mathbb{H}_i(f_n) \cap \bigcap_{n \geq 1} \mathbb{H}_i(f_n^{-1}).
\]

\begin{figure}\label{GEO}
    \centering
    \includegraphics[width=0.7\linewidth]{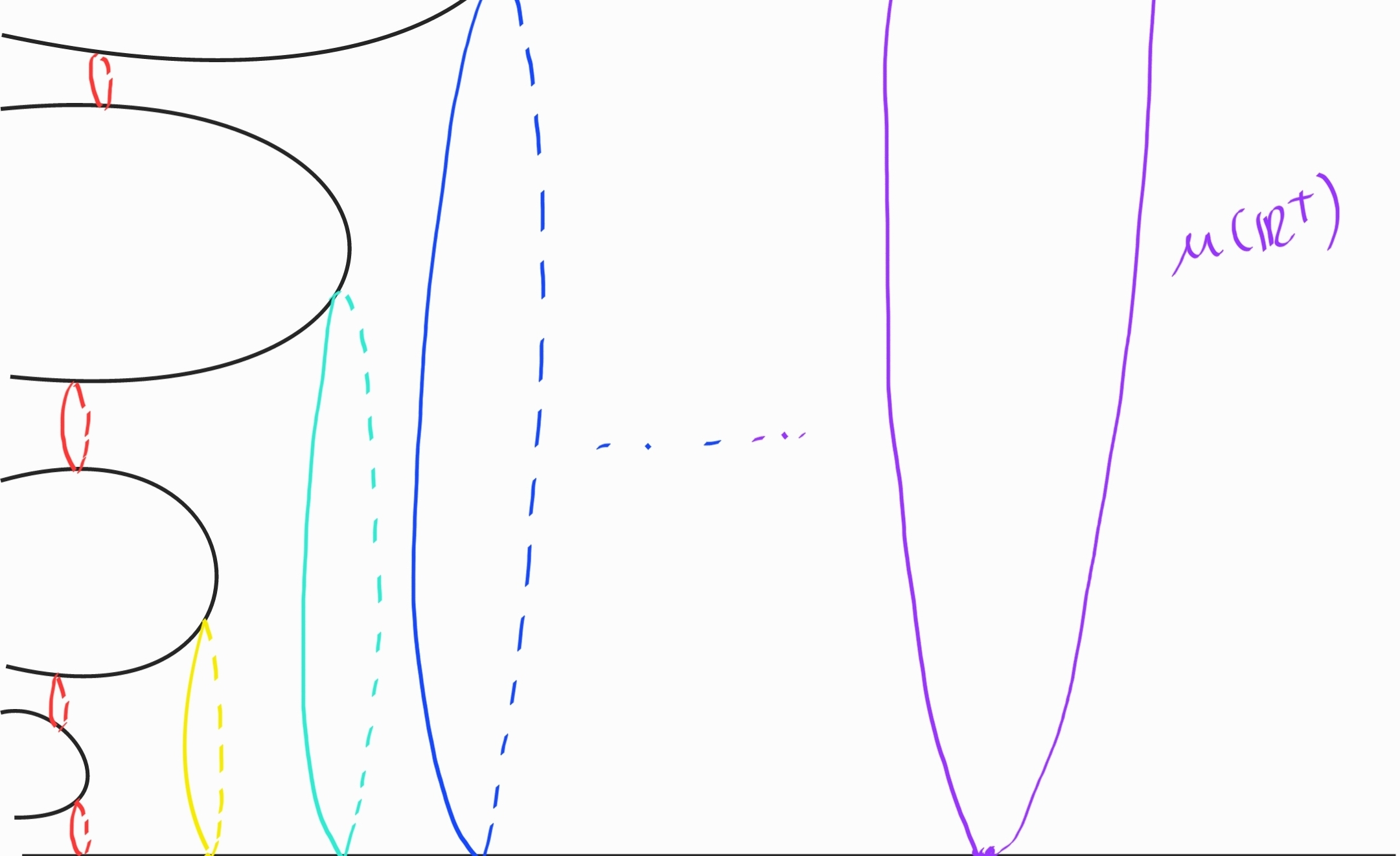}
    \caption{untwisted flute (Surface)}
    
\end{figure}
\newpage
\subsection{Asymptotic fineness of an infinite quasi-minimizing ray}

Let $u$ be an element in the unit tangent bundle of this untwisted flute $T$.\\ Let $\tilde{u}$ be the lift of $u$ in $T^{1}\mathbb{H}.$
\begin{dfn}
A geodesic ray $u(\mathbb{R}^+)$ on a hyperbolic surface $S$ is said to be \emph{quasi-minimizing} if there exists a constant $C \geq 0$ such that for all $t > 0$ one has $d(u(0), u(t)) \geq t - C.$\\ Moreover, the geodesic ray $u(\mathbb{R}^+)$ is said to be \emph{minimizing} if for all $t > 0$ one has $d(u(0), u(t)) = t.$
   
\end{dfn}

\begin{dfn}
An \emph{infinite quasi-minimizing half-geodesic} of a hyperbolic surface $S = \Gamma \backslash \mathbb{H}$ is a half-geodesic whose initial vector lies in the non-wandering set of the horocycle flow and is not $h_{\mathbb{R}}$-periodic.\\ Equivalently, the endpoint $\tilde{u}(+\infty)$ of a lift of $u(\mathbb{R}_+)$ in $\mathbb{H}$ belongs to $ L(\Gamma) \setminus \big( L_h(\Gamma) \cup L_p(\Gamma) \big).$
\end{dfn}

\begin{prp}{\cite{BellisThesis}}
Let $\tilde{u}_0$ be in $T^1\mathbb{H}$ based at $i$ and directed towards $\infty$.  \\
The half-geodesic $\tilde{u}_0(\mathbb{R}^+)$ projects down to  $S := \Gamma \backslash \mathbb{H}$  as an infinite quasi-minimizing half-geodesic $u_0(\mathbb{R}^+)$.
\end{prp}

\begin{prp}{\cite{BellisThesis}}
The half-geodesic $\widetilde{u}_0(\mathbb{R}_+) = [i, \infty)$ in $T^1\mathbb{H}$ projects onto $T$ as a quasi-minimizing half-geodesic of infinite asymptotic finess.
\end{prp}

\subsection{No-closed nor-dense Orbit}
\begin{prp}
On the unit tangent bundle of the untwisted flute $T$, the half-geodesic $u_0(\mathbb{R}_+)$ is infinite quasi-minimizing with infinite asymptotic fineness, and the horocycle orbit $h_\mathbb{R} u_0 $ is not closed.\\
\end{prp}
\begin{proof}
Let us denote by $\tilde{v}$ the element of $T^1 T$ whose lift $\tilde{v}$ is based at $i$ and directed towards $0$. That is, $v$ is the vector opposite to $u$.
We will show that $v \in \overline{h_\mathbb{R} u} \setminus h_\mathbb{R} u$. Observe that for every $n$, we have $\gamma_n^{-1} i = p_n$. Thus, $B_\infty(\gamma_n^{-1} i, i) =\epsilon_n.$
Therefore, the sequence of distinct elements $(\gamma_n)_{\mathbb{N}}$ satisfies $\lim_{n \to +\infty} B_\infty(\gamma_n^{-1} i, i) = 0.$ Moreover, $\lim_{n \to +\infty} \gamma_n \infty = 0.$
\\
According to the horocycle convergence, this implies that indeed $v \in \overline{h_\mathbb{R} u}$.\\
\end{proof}

\subsection{Let us show that $T_u = \{ t \in \mathbb{R} : g_t u \in \overline{h_{\mathbb{R}} u} \} = \{0\}$.}

Suppose that there exists $t_1\neq 0$ such that $g_{t_{1}}u \in\overline{h_{\mathbb{R}}u}$ with $u\in\Omega_{h}$.
Then, according to \ref{tunv}, there exists a sequence $(\gamma_n)\subset\Gamma$ of isometries, all distinct from the identity, satisfying:

\begin{enumerate}
    \item $\gamma_n(\infty) \to \infty $,
    \item $B_\infty(\gamma_{n}^{-1}(i), i) \to t_1 $.
\end{enumerate}

Writing $\gamma_n(z) = \frac{a_n z + b_n}{c_n z + d_n}$  with $ a_n d_n - b_n c_n = 1 $, and using conditions $(1)$ and $(2)$, we can extract subsequences $(a_n)$ and $(c_n)$ converging respectively to $a \neq 0 $ and $0$.

By applying Lemma \ref{T} , we can deduce that the sequence $(b_n)$ also converges to $0$.
The relation $a_n d_n - b_n c_n = 1$ ensures that the sequence $(d_n)$ converges to $\frac{1}{a}$.

Since the group $\Gamma$ is discrete, we deduce that $\gamma_n$ is constant for sufficienttly large $n$, which implies that $a=d=1$. This contradicts our hypothesis.

\section{properties of $T_{u}$}

The set $\tilde{T}_u = \{ t \in \mathbb{R}^+ \mid g_t u \in\overline{ h_{\mathbb{R}}(u)} \}$ plays an important role in describing the non-minimal sets for the horocycle flow. We will show that this set is a closed semigroup of $\mathbb{R}^+$.

\begin{prp}
Let $S$ be a geometrically infinite hyperbolic surface and $T^1S$ its unit tangent bundle. Let $u \in \Omega_h$. Then the set $\tilde{T_u}$ is a semigroup of $\mathbb{R}^+$.
\end{prp}

\begin{proof}
Let $t_1, t_2 \in \tilde{T_u}$. By definition, we have 
$g_{t_1} u \in \overline{h_{\mathbb{R}}(u)}.$
Since $h_{\mathbb{R}}(u)$ is invariant under the horocycle flow, for any $s \in \mathbb{R}$, we also have $h_s g_{t_1} u \in \overline{h_{\mathbb{R}}(u)}.$
On the other hand, since $t_2 \in \tilde{T_u}$, there exists a sequence $(s_n)_n$ of real numbers such that $h_{s_n} u \to g_{t_2} u.$
Using the quasi-commutation relation between the geodesic and horocycle flows, we get $
g_{t_1} h_{s'_n} u = h_{s_n} g_{t_1} u \in h_{\mathbb{R}}(u)$
for some sequence $(s'_n)_n$. From the convergence $h_{s_n} u \to g_{t_2} u$, we deduce that $g_{t_1} h_{s_n} u \to g_{t_1 + t_2} u,
$ which means there exists a sequence \((a_n)_n\) of real numbers such that $h_{a_n} g_{t_1} u \to g_{t_1 + t_2} u.$
But since for all \(s \in \mathbb{R}\), \(h_s g_{t_1} u \in \overline{h_{\mathbb{R}}(u)}\), the sequence \(h_{a_n} g_{t_1} u\) lies in $\overline{h_{\mathbb{R}}(u)}$ and converges to \(g_{t_1 + t_2} u\). Hence,
$g_{t_1 + t_2} u \in \tilde{T_u}.$
Therefore, \(\tilde{T_u}\) is a semigroup.
\end{proof}
\newpage
\begin{prp}
The set $\tilde{T}_u = \{ t \in \mathbb{R}^+ \mid g_t u \in \overline{h_{\mathbb{R}}(u)} \}$ is a non-empty closed subset of $\mathbb{R}^+.$
\end{prp}

\begin{proof}
Recall that the non-wandering set for the horocycle flow is given by $\Omega_h = \{ u \in T^1 S \mid \tilde{u}(+\infty) = +\infty\in\Lambda(\Gamma) \}.
$ Suppose that $t_0 \in \mathbb{R}^+ \setminus T_u$, i.e.,$g_{t_0} u \notin \overline{h_{\mathbb{R}}(u)}.
$ Since $\Omega_h \setminus \overline{h_{\mathbb{R}}(u)}$ is open, there exists a ball \(B(g_{t_0} u; \varepsilon)\) of center $g_{t_0} u$ and radius $\varepsilon > 0$ such that $B(g_{t_0} u;\varepsilon) \subset \Omega_h \setminus \overline{h_{\mathbb{R}}(u)}.$ In particular, $B(g_{t_0} u; \varepsilon) \cap \Omega_h \neq \emptyset.$
It follows that there exists $\eta > 0$ such that for all $t \in (\eta - t_0, \eta + t_0),$ we have $g_t u \notin \overline{h_{\mathbb{R}}(u)}.$
Thus, there exists an open interval containing $t_0$ in which no element belongs to $\tilde{T_u}$. Hence, $\tilde{T_u}$ is closed.
\end{proof}

\section{Example in the case that $T_{u}$ is infinite}

We consider the group $\Gamma = \langle h_{p_n} \rangle_{n \in \mathbb{N}^*}$, an infinite Schottky group, with $h_{p_n}(z)=\frac{a_{p_n} z + b_{p_n}}{c_{p_n} z + d_{p_n}}, $ where $\delta>1$ is a fixed real number. For every integer $n \geq 1$, we define: $a_{p_n} = \delta+\frac{ 2p_n}{p_n^2 + 1}, \quad b_{p_n} = p_n + (p_n^2 + 1)\delta, \quad c_{p_n} = \frac{1}{p_n}, \quad d_{p_n} = \frac{p_n^2 + 1}{p_n}.$ For every $n \geq 1$, we verify that: $a_{p_n} d_{p_n} - b_{p_n} c_{p_n} = 1.$
The sequence $(p_n)_{n \geq 1}$ of natural numbers satisfies the recurrence relation:
$p_1 = 1 + \left\lfloor \frac{\delta - 1}{2} \right\rfloor, \quad
p_{n+1} = 1 + \left\lfloor \frac{(\delta + 1)p_n}{\delta - 1} \right\rfloor= 1 + p_n + \left\lfloor \frac{2p_n}{\delta - 1} \right\rfloor,$ where \( \lfloor \cdot \rfloor \) denotes the integer part (see \cite{SyGaye2025} for more details).
The surface $S = \mathbb{H}/\Gamma$ obtained is a twisted flute. Moreover, there exists a half-ray $u \in \Omega_h $ that is quasi-minimizing.
The surface has infinite asymptotic fineness (see \cite{SyGaye2025}).

The horocyclic orbit $h_{\mathbb{R}} u$ is irregular and non-recurrent. The following lemma shows the non-minimality of the closure of this orbit.

\begin{lem}
Let $\Gamma = \langle h_{p_n} \rangle_{n \in \mathbb{N}^{*}}$ be a geometrically infinite Fuchsian group, where $h_{p_n}(z) = \frac{a_{p_n}z + b_{p_n}}{c_{p_n}z + d_{p_n}}$. Let $S = \mathbb{H}/\Gamma$ be the associated hyperbolic surface, and let $u \in T^{1}S$. Then there exists an unbounded sequence of times $(t_n)_n$ such that $g_{t_n}u \in \overline{h_{\mathbb{R}}u}$.
\end{lem}

\begin{proof}
Let $\Gamma$ be the Fuchsian group generated by the $h_k$, and let $W = \{h_{p_n}, h_{p_n}^{-1}\}$ be the alphabet. Let $u \in T^1 S$, and let $\tilde{u}(\mathbb{R}^+)$ be the lift of $u(\mathbb{R}^+)$ to $T^1 \mathbb{H}$.  

Without loss of generality, we may conjugate the group so that $u(\infty) = +\infty$.  
We consider the reduced word $\gamma_{n,k} = h_{p_n} h_{p_n^2} h_{p_n^4} \ldots h_{p_n^{2^k}}$ for all $k \geq 1$. We have:

\[
h_{p_n} h_{p_n^2} h_{p_n^4} \ldots h_{p_n^{2^k}}(\infty) = h_{p_n}(\infty) \left[1 - \frac{1}{a_{p_n} c_{p_n} h_{p_n^2} h_{p_n^4} \ldots h_{p_n^{2^k}}(\infty) + a_{p_n} d_{p_n}}\right].
\]

Since $h_{p_n}(\infty) = \frac{a_{p_n}}{c_{p_n}} = \delta p_n + \frac{2p_n^2}{p_n^2 + 1}$ and $p_n \to +\infty$, it follows that $\lim \gamma_{n,k}(\infty) = +\infty$.  

Moreover, 
\[
a_{p_n} c_{p_n} h_{p_n^2} h_{p_n^4} \ldots h_{p_n^{2^k}}(\infty) + a_{p_n} d_{p_n} > a_{p_n} d_{p_n},
\]
and 
\[
a_{p_n} d_{p_n} = \delta \frac{p_n^2 + 1}{p_n} + 2.
\]

According to Proposition $1.11.5$, we obtain:
\begin{align*}
B_{\infty}(\gamma_{n,k}^{-1}i, i) &= B_{\infty}(h_{p_n^{2^k}}^{-1} h_{p_n^{2^{k-1}}}^{-1} \ldots h_{p_n^2}^{-1} h_{p_n}^{-1}i, i) \\
&= \sum_{l=1}^{k} B_{h_{p_n^{2^l}} \ldots h_{p_n^{2^k}}\infty}(h_{p_n^{2^{l-1}}}^{-1}i, i) + B_{\infty}(h_{p_n^{2^k}}^{-1}i, i) \\
&= \sum_{l=1}^{k} B_{\infty}(h_{p_n^{2^{l-1}}}^{-1}i, i) + \sum_{l=1}^{k} \log \left|\frac{h_{p_n^{2^{l-1}}}^{-1}i - h_{p_n^{2^l}} \ldots h_{p_n^{2^k}}\infty}{i - h_{p_n^{2^l}} \ldots h_{p_n^{2^k}}\infty}\right|^2 \\
&\quad + B_{\infty}(h_{p_n^{2^k}}^{-1}i, i).
\end{align*}

We now analyze the limits:

\[
\lim_{n \to +\infty} B_{\infty}(h_{p_n^{2^k}}^{-1}i, i) = \lim_{n \to +\infty} -\log \operatorname{Im}(h_{p_n^{2^k}}^{-1}i) = \lim_{n \to +\infty} \log(a_{p_n^{2^k}}^2 + c_{p_n^{2^k}}^2) = \log \delta^2.
\]

Similarly,
\[
\lim_{n \to +\infty} \sum_{l=1}^{k} B_{\infty}(h_{p_n^{2^{l-1}}}^{-1}i, i) = k \log \delta^2.
\]

Finally,
\[
\lim_{n \to +\infty} \sum_{l=1}^{k} \log \left|\frac{h_{p_n^{2^{l-1}}}^{-1}i - h_{p_n^{2^l}} \ldots h_{p_n^{2^k}}\infty}{i - h_{p_n^{2^l}} \ldots h_{p_n^{2^k}}\infty}\right|^2 = k \log\left(\frac{(\delta + 1)^2}{\delta^2}\right),
\]
since
\[
\lim_{n \to +\infty} \frac{h_{p_n^{2^{l-1}}}^{-1}i}{h_{p_n^{2^l}} \ldots h_{p_n^{2^k}}\infty} = -\frac{1}{\delta}.
\]

In conclusion,
\[
\lim_{n \to +\infty} B_{\infty}(\gamma_{n,k}^{-1}i, i) = k \log(\delta^2) + \log(\delta^2) + k \log\left(\frac{(\delta + 1)^2}{\delta^2}\right) = 2 \log\left(\delta(\delta + 1)^k\right).
\]

Thus, there exists an unbounded sequence of times $(t_k)_k$ defined by $t_k = 2 \log\left(\delta(\delta + 1)^k\right)$ such that $g_{t_k}u \in \overline{h_{\mathbb{R}}u}$.
\end{proof}



\textbf{Acknowledgments}\\
This work was done during the author's stay in the Institut of Henri Poincaré  (UAR 839 CNRS-Sorbonne Université), LabEx CARMIN (ANR-10-LABX-59-01). The author expresses gratitude to Abdoul Karim Sané for the productive discussion and the support of NLAGA and AFRIMath.

\end{document}